\documentclass[12pt,bothside,a4paper]{article}
\usepackage{amsthm}
\usepackage{amsfonts}
\usepackage{mathrsfs}
\usepackage{amsmath}
\usepackage{galois}
\usepackage[numbers,sort&compress]{natbib}
\usepackage[top=1in,bottom=1in,left=1.25in,right=1.25in]{geometry}
\usepackage{amssymb}
\usepackage{xcolor}

\usepackage{bm}
\theoremstyle{plain}

\theoremstyle{definition}
\newtheorem{definition}{Definition}[section]
\theoremstyle{lemma}
\newtheorem{lemma}[definition]{Lemma}
\newtheorem{theorem}[definition]{Theorem}
\theoremstyle{remark}
\newtheorem{remark}[definition]{Remark}

\numberwithin{equation}{section}
\allowdisplaybreaks[4]
\usepackage{fontenc}
\usepackage{inputenc}
\usepackage{authblk}

\def\E{\mathbb E}
\def\C{\mathcal C}
\def\t{\partial_t}
\def\T{\mathbb T^2}
\def\d{\mathrm{d}}
\def\Z{\mathbb Z^2}
\begin{document}
\date{}
\pagestyle{plain}
\title{Convergence rate for Galerkin approximation of the stochastic Allen-Cahn equations on 2D torus }
\author{Ting~Ma$^a$,~Rongchan~Zhu$^{b,c,}$
\thanks{Research supported in part by NSFC (No.11671035). Financial support by the DFG through the CRC
1283 ''Taming uncertainty and profiting from randomness and low regularity in analysis, stochastics and their
applications'' is acknowledged.}%
\thanks{Corresponding author}%
\thanks{Email address: matingting2008@yeah.net(T. Ma), zhurongchan@126.com(R. C. Zhu)}}
\affil[a]{College of Mathematics, Sichuan University, Chengdu 610065, China}
\affil[b]{Department of Mathematics, Beijing Institute of Technology, Beijing 100081, China}
\affil[c]{Department of Mathematics, University of Bielefeld, D-33615 Bielefeld, Germany}
\maketitle \underline{}
\begin{abstract}
In this paper we discuss the  convergence rate for Galerkin approximation of the stochastic Allen-Cahn equations driven by space-time white noise on $\T$. First we prove that the  convergence  rate for stochastic 2D heat equation is  of order $\alpha-\delta$ in Besov space $\C^{-\alpha}$ for $\alpha\in(0,1)$ and $\delta>0$ arbitrarily small. Then we obtain the convergence rate for Galerkin approximation of the stochastic Allen-Cahn equations of order $\alpha-\delta$ in $\C^{-\alpha}$ for $\alpha\in(0,2/9)$ and $\delta>0$ arbitrarily small.
\end{abstract}
\setlength{\baselineskip}{0.25in}

\section{Introduction}

~~~~In this paper we study the  convergence rate for Galerkin approximation of the stochastic Allen-Cahn equations (see \eqref{initial-Eq}) on 2D torus driven by space-time white noise.
Such stochastic partial differential equations (SPDEs)   contain superlinearly growing nonlinearities in their coefficients  and in general they can not be solved explicitly. It is a quite active area  to design and analyze approximation algorithms, which solve
SPDEs with superlinearly growing nonlinearities approximatively.
Galerkin  approximation has been  applied to study solutions to SPDEs   (see \cite{LR13}). It could also be regarded as one of the basic finite elements methods in spatial approximation (see e.g. \cite{BGJK2017,BHJKL19,GM05,GSS16,HJ14,HJS17,JP15,Y05}).
 There are also a lot of work on other  spatial approximation methods including Fourier method,  piecewise linear approximation,  finite elements methods (see e.g. \cite{BH19,GM05,HM16,ZZ2015b}), and   finite differences approximation (see e.g. \cite{BHJKL19,GM05,DG01,G98,G99,FKLL16}).

Approximations to SPDEs driven by trace-class Wiener process have been studied a lot in the literatures (see e.g. \cite{BH19,FKLL16,GSS16,JP15,LR13,CH18,HJ14,KLL15,Y05}).
For SPDEs driven by space-time white noise, we refer to
\cite{HJS17,AS2017,BGJK2017,DG01,G98,G99,HM16} and the references therein for the convergence of spatial approximations and
 refer to \cite{BJ16,HJS17,AS2017,BGJK2017,DG01,G98,G99} and the references therein for the convergence of temporal approximations.
Furthermore, in  the references \cite{BJ16,BGJK2017,DG01,G99,HM16} the convergence rates of spatial and  temporal approximations  were also obtained.
To be more specific, in \cite{G99} pointwise estimates were considered for the stochastic quasi-linear parabolic PDEs with locally bounded coefficients on one dimensional space driven by  multiplicative space-time white noise  and the convergence rate of  order $1/2$ was obtained. In \cite{DG01} the rate of convergence for the  stochastic   heat equation was further improved to $1$ for  additive noise by   estimating space averages of the solution rather than pointwise estimates.
S. Becker, B. Gess, A. Jentzen and P. E. Kloeden
in \cite{BGJK2017} obtained the convergence of full-discrete  approximations with rates of order $1/2-\epsilon,\epsilon>0$ in space and $1/4-\epsilon,\epsilon>0$ in time for  stochastic Allen-Cahn equations driven by  space-time white noise  on one dimensional space.
Using the rough path theory (see \cite{L98,G04,G10}), M. Hairer and K. Matetski in \cite{HM16} showed convergence of spatial approximations with rate of order $1/2-\epsilon$  for  Burgers type SPDEs driven by space-time white noise on one dimensional space.

All the references we conclude above are  about SPDEs driven by space-time white noise on one dimensional case. For the higher dimensional cases,
Yan in \cite{Y05} obtained the convergence rates of spatial and temporal approximations for linear SPDEs driven by space-time white noise on $d~(d=1,2,3)$ dimensional space.
In  \cite{TW2018,TW20182,ZZ2015,MW2010} the authors studied the stochastic Allen-Cahn equations (the dynamical $\Phi^4_d$, $d>1$ model) and obtained convergence with no description of the convergence rates.
To the best of our knowledge, there exists no result in the literature, which establishes the convergence rates for numerical approximation of SPDEs  with superlinearly growing nonlinearities  driven by  space-time white noise  on high-dimensional space. In this paper we study the convergence rates for  stochastic 2D  Allen-Cahn equations  driven by space-time white noise.

For spatial dimension $d\geq2$,  stochastic Allen-Cahn equations \eqref{initial-Eq} driven by space-time white noise are  ill-posed in the classical sense and the main difficulty in this case is that the noise $\xi$ and hence the solution $X$ are so singular that the non-linear terms are not well-defined in the classical sense. In two spatial dimensions, this problem
was previously treated in \cite{AR91,DPD03}.
These kinds of singular SPDEs have received a lot of attention recently (see e.g. \cite{ Hai14,GIP13}) and the renormalization is required.
In this paper we  define the approximating equations of the stochastic Allen-Cahn equations via Galerkin projection.  As explained above,  we need to consider the convergence of Galerkin approximation using renormalization (see \eqref{app-Eq}) and we obtain the convergence rate for this approximation.
\subsection{Statement and main results} \label{Statement of the main results}
Consider the equation
\begin{equation}
\label{initial-Eq}
\left\{
 \begin{aligned}
\t X&= \Delta X+:F(X):+\xi,\text{~~on~} (0,\infty)\times\T,\\
X(0)&=X_0 \text{~~on~} \T.
 \end{aligned}
 \right.
\end{equation}
Here $\mathbb T^2$ is a torus of size $1$ in $\mathbb R^2$, $\xi$ is  space-time white noise on $\mathbb R^+\times\T$ (see Definition \ref{WHITE-NOISE}), $\Delta$ is the Laplacian with periodic boundary conditions on $L^2(\T)$. $:F(v)::=\sum_{j=0}^3a_jv^{:j:}$, $a_3<0$, where $v^{:0:}=1,v^{:1:}=v$ and $v^{:2:},v^{:3:}$, the Wick powers of $v$, are defined by approximation in Section \ref{set-Wick powers}. The initial value $X_0\in\C^{-\alpha}$, $\alpha\in(0,1)$, which is defined in Section \ref{subsec-BS}.
Following \cite{DPD03,MW2010}, we say that $X$ solves the equation (\ref{initial-Eq}) if
\begin{equation}\label{sum}
 X=Y+\bar Z,
\end{equation}
where $\bar Z$ satisfies  the following   stochastic heat equation
\begin{equation}
\label{ini1-Eq}
\left\{
 \begin{aligned}
\t \bar Z&= A \bar Z+\xi, \text{~~on~} (0,\infty)\times\T,\\
\bar Z(0)&=X_0 \text{~~on~} \T,
 \end{aligned}
 \right.
\end{equation}
with $A=\Delta-{\rm I}$ and $Y$ solves
\begin{equation}
\label{ini2-Eq}
\left\{
 \begin{aligned}
\t Y&=\Delta Y+\Psi(Y,\underline{ \bar Z}), \text{~~on~} (0,\infty)\times\T,\\
Y(0)&=0,\text{~~on~} \T,
 \end{aligned}
 \right.
\end{equation}
with $\underline{z}:=(z,z^{:2:},z^{:3:})$ and
\begin{equation}\label{wick-YZ}
\Psi(y,\underline{z}):= \sum^3_{j=0} a_j \sum ^j_{k=0} \binom{j}{k} y^k {z}^{:j-k:}+{z}.
 \end{equation}
We interpret (\ref{ini2-Eq}) in the mild sense, i.e.  $Y$ solves (\ref{ini2-Eq}) if for every $t\geq 0$,
\begin{equation}\label{mild-2}
 Y_t=\int^t_0 e^{{(t-s)}\Delta}\Psi(Y_s,\underline{\bar Z}_s)ds.
\end{equation}
Next, we consider the Galerkin approximations of \eqref{ini1-Eq}: for $N\geq 1$,
\begin{equation}
\label{ini1-Eq_N}
\left\{
 \begin{aligned}
\t \bar Z^N&= P_N A \bar Z^N+\xi^N,  \text{~~on~} (0,\infty)\times\T,\\
\bar Z^N(0)&=P_NX_0, \text{~~on~} \T,
 \end{aligned}
 \right.
\end{equation}
with $P_N$, the projection operators on $L^2(\T)$,  given in \eqref{pro-ope}, $\xi^N:=P_N\xi$. By \cite[(2.3)]{TW2018} and \cite[p.4]{DPD03} there exist constants $\mathfrak{R}^{N}$ given in \eqref{renor-cosnt}, which diverge logarithmically as $N$ goes to $\infty$, such that
\begin{equation}\label{wick-ZN-2-3}
  (\bar Z^N)^{:2:}:=(\bar  Z^{N})^2-\mathfrak{R}^{N},~~~~(\bar Z^N)^{:3:}:=( \bar Z^{N})^3-3\mathfrak{R}^{N}\bar Z^{N}
\end{equation}
converge in $L^p(\Omega;C([0,T];\C^{-\alpha}))$ with $\alpha\in(0,1)$, $p>1$ to non-trivial limits denoted by $\bar Z^{:2:}$ and $\bar Z^{:3:}$, respectively. We refer to Section \ref{SEC3} for  details.
Then we consider the following  Galerkin approximation of \eqref{initial-Eq}
\begin{equation}
\label{app-Eq}
\left\{
 \begin{aligned}
&\t X^N= P_N \Delta X^N+ P_N F^N(X^N)+\xi^N, \text{~~on~} (0,\infty)\times\T,\\
&X^N(0,\cdot)= P_N  X_0,\text{~~on~} \T,
 \end{aligned}
 \right.
\end{equation}
with
\begin{equation*}\label{app-FN}
F^{N}(v)=a_3v^{3}+a_2v^{2}+a_1^{N}v+a_0^{N},
\end{equation*}
\[
a_1^{N}:=a_1-3a_3\mathfrak{R}^{N},~a_0^{N}:=a_0-a_2\mathfrak{R}^{N}.
\]
Then $X^N$ solves \eqref{app-Eq} if
\begin{equation}\label{sumN}
X^N=\bar Z^N+Y^N
 \end{equation}
 with $\bar Z^N$ solving \eqref{ini1-Eq_N} and $Y^N$ solving
\begin{equation}
\label{ini2-Eq_N}
\left\{
 \begin{aligned}
\t Y^N&= P_N\Delta Y^N+P_N\Psi(Y^N,\underline{\bar Z}^N),  \text{~~on~} (0,\infty)\times\T,\\
Y^N(0)&=0,~\text{~~on~} \T,
 \end{aligned}
 \right.
\end{equation}
 in the mild sense that for every $t\geq 0$,
\begin{equation}\label{mild-2-N}
 Y_t^N=\int^t_0P_N e^{{(t-s)}\Delta}\Psi(Y^N_s,\underline{\bar Z}^N_s)ds,
\end{equation}
with $\underline{\bar Z}^N=({\bar Z}^N, ({\bar Z}^N)^{:2:},({\bar Z}^N)^{:3:})$ given in \eqref{wick-ZN-2-3}.


We  mainly discuss the convergence rates for the stochatic Allen-Cahn equations \eqref{initial-Eq}. The main results obtained is as below. See Theorem \ref{main-rate} for more details.
\begin{theorem}
 Let $\alpha\in(0,2/9),\gamma'>3\alpha/2$ and  $X_0\in\C^{-\alpha}$. Let $X$,$X^N$ denote the solutions to equations \eqref{initial-Eq} and \eqref{app-Eq} on $[0,T]$ with initial values $X_0$ and $P_NX_0$, respectively.  Then for any $\delta>0$
\[
 \lim_{N\rightarrow\infty}\mathbb P \Big(\sup_{t\in[0,T]}t^{\gamma'} \|X_t-X_t^N\|_{{-\alpha}}\gtrsim N^{\delta-{\alpha}}\Big)=0.
\]
\end{theorem}

\subsection{Structure}
This paper is organized as follows: In Section \ref{subsec-BS} we collect results related to Besov spaces.
In Section \ref{SEC3} we  show the convergence rates for  linear equation \eqref{ini1-Eq}.
In Section \ref{Pathwise-error} we prove the main results, i.e., the convergence rates for stochastic Allen-Cahn equations \eqref{initial-Eq}.

\begin{section}{Besov spaces and preliminaries}\label{subsec-BS}
We first recall Besov spaces from \cite{ZZ2015,MW2010}.  For general theory we refer to  \cite{BCD2011,Tri78,Tri06}. Throughout the paper, we use the notation $a \lesssim b$ if there exists a constant $c > 0$ independent of
the relevant quantities such that $a \leq cb$, we also use the notation $a\gtrsim b$ if $b\lesssim a$, and  use the notation $a\backsimeq b$ if $a \lesssim b$ and $b\lesssim a$. For $p\in[1,\infty]$, let $L^p(\mathbb T^d)$ denote the usual $p$ integrable space on $\mathbb T^d$ with its norm denoted by $\|\cdot\|_{L^p}$.  The space of Schwartz functions on $\mathbb T^d$ is denoted by $\mathcal S (\mathbb T^d)$ and its dual, the space of tempered distributions is denoted by $\mathcal S'(\mathbb T^d)$. The space of real valued infinitely differentiable functions is denoted by $C^{\infty}(\mathbb T^d)$. For any function $f$ on $\mathbb T^d$, let ${\rm{supp}} (f)$ denote its support.

Consider the orthonormal basis $\{e_m\}_{m\in\mathbb Z^d}$ of trigonometric functions on $\mathbb T^d$
\begin{equation}\label{basis}
  e_m(x):=e^{\iota2\pi m \cdot x},  x\in\mathbb T^d,
\end{equation}
we write $L^2(\mathbb T^d)$ with its inner product $ \langle \cdot,\cdot\rangle$ given by
 \[
    \langle f,g\rangle=\int_{\mathbb T^d}f(x)\bar g(x)\d x,~~f,g\in L^2(\mathbb T^d).
 \]
Then for any $f\in L^2(\mathbb T^d)$, we denote by $\mathcal F f$ or $ \hat{f}$ its Fourier transform

\[
   \hat{f}(m)=\langle f, e_m\rangle=\int_{\mathbb T^d} e^{-\iota2\pi m \cdot x}f(x)\d x,~~m\in\mathbb Z^d.
\]
For  $N\geq 1$, we define $P_N$ the projection operators from $L^2(\mathbb T^d)$ onto the space spanned by $ \{e_m,~|m|\leq N\}$ with $|m|:=\sqrt{m_1^2+m_2^2+\cdots+m^2_d},~m=(m_1,\ldots,m_d)\in\mathbb Z^d$, i.e.
\begin{equation}\label{pro-ope}
P_N f=\sum_{m:|m|\leq N}\langle f, e_m\rangle ~e_m,~f\in L^2(\mathbb T^d).
\end{equation}
For $\zeta\in\mathbb R^d$ and $r > 0$ we denote by $B(\zeta, r)$ the ball of radius $r$ centered at $\zeta$ and let the annulus $\mathcal A:=B(0,\frac 8 3)\setminus B(0,\frac 3 4)$. According to \cite[Proposition 2.10]{BCD2011}, there exist nonnegative radial functions $\chi, \theta\in\mathcal D(\mathbb R^d)$, the space of real valued infinitely differentiable functions of compact support on $\mathbb R^d$, satisfying

i. ${\rm {supp}}(\chi)\subset B(0,\frac 4 3),~{\rm {supp}}(\theta)\subset\mathcal A$;

ii. $\chi(z)+\sum_{j\geq 0}\theta(z/{2^j})=1$ for all $z\in\mathbb R^d$;

iii. ${\rm {supp}}(\chi)\cap {\rm {supp}}(\theta(\cdot/{2^j}))=\emptyset$ for $j\geq1$ and ${\rm {supp}}(\theta(\cdot/{2^i}))\cap {\rm {supp}}(\theta(\cdot/{2^j}))=\emptyset$ for $|i-j|>1$.

$(\chi,\theta)$ is called a dyadic partition of unity. The above decomposition can be applied to distributions on the torus (see \cite{S85,SW71}).  Let
\[
   \chi_{-1}:=\chi,~~\chi_{j}:=\theta(\cdot/{2^j}),~j\geq0,
\]
we have ${\rm {supp}}(\chi_{j})\subset \mathcal A_{2^j}:=2^j \mathcal A$ for every $j\geq0$ and $\mathcal A_{2^{-1}}\subset B(0,\frac 4 3)$. For  $f\in C^\infty(\mathbb T^d)$, the $j$-Littlewood-Paley block is  defined as
\begin{equation}\label{LPBs}
\Delta_{j}f(x)=\sum_{m\in\mathbb Z^d}\chi_{j}(m)\hat{f}(m) e^{\iota2\pi m\cdot x},~j\geq-1.
\end{equation}
 It is noticeable that \eqref{LPBs} is equivalent to the equality
\begin{equation}\label{LPBs2}
  \Delta_{j}f=\eta_j \ast f,~~j\geq-1,
\end{equation}
where
\[
\eta_j\ast f(\cdot)=\int_{\mathbb T^d} \eta_j (\cdot-x)f(x)dx,~~
\eta_j(x):=\sum_{m\in\mathbb Z^d}\chi_{j}(m) e^{\iota2\pi m \cdot x}.
\]
For $\alpha\in\mathbb R,~p,q\in[1,\infty]$, we define the Besov space on $\mathbb T^d$ denoted by $\mathcal B_{p,q}^{\alpha}(\mathbb T^d)$  as
 the completion of $C^\infty(\mathbb T^d)$ with respect to the norm (\cite[Proposition 2.7]{BCD2011})
\begin{equation}\label{B-SPACE}
  \|u\|_{\mathcal B_{p,q}^{\alpha}(\mathbb T^d)}:=(\sum_{j\geq-1} 2^{j\alpha q}\|\Delta_j u\|^q_{L^p})^{1/q},
\end{equation}
with the usual interpretation as $l^{\infty}$ norm in case $q=\infty$.
Note that for $p,q\in[1,\infty)$
\begin{equation*}
\begin{aligned}
   \mathcal B_{p,q}^{\alpha}(\mathbb T^d)&=\{u\in\mathcal S'(\mathbb T^d):~\|u\|^q_{ \mathcal  B_{p,q}^{\alpha}(\mathbb T^d)}<\infty\},\\
  \mathcal B_{\infty,\infty}^{\alpha}(\mathbb T^d)&\varsubsetneq\{u\in\mathcal S'(\mathbb T^d):~\|u\|^q_{\mathcal  B_{\infty,\infty}^{\alpha}(\mathbb T^d)}<\infty\}.
\end{aligned}
\end{equation*}
 Here we choose Besov spaces as completions of smooth functions on the torus,
which ensures that the Besov spaces are separable and has a lot of advantages for our analysis
below. In the following we give estimates on the torus for later use.

We recall the following Besov embedding theorems on the torus  (cf. \cite[Theorem 4.6.1]{Tri78}, \cite[Lemma A.2]{GIP13}, \cite[Proposition 3.11,Remark 3.3]{MW2010}).
\begin{lemma} \label{V-SUB-H}{\rm(Besov embedding)}\label{Bes-Smooth}
{\rm (i)} Let $\alpha\leq \beta\in\mathbb R$,  $p,q\in[1,\infty]$. Then $ \mathcal B_{p,q}^{\beta}(\mathbb T^d)$ is continuously embedded   in $ \mathcal B_{p,q}^{\alpha}(\mathbb T^d)$.

{\rm (ii)}~ Let $1\leq p_1\leq p_2\leq \infty$, $1\leq q_1\leq q_2\leq \infty$, and let $\alpha\in \mathbb R$. Then $\mathcal B_{p_1,q_1}^{\alpha}(\mathbb T^d)$ is continuously embedded in  $\mathcal B_{p_2,q_2}^{\alpha-d(1/p_1-1/p_2)}(\mathbb T^d)$.

\end{lemma}

We describe the following Schauder estimates, i.e. the smoothing effect of the heat flow, as measured in Besov spaces (cf. \cite[Propositions 3.11,3.12]{MW2010}, \cite[Lemmas A.7,A.8]{GIP13}).

\begin{lemma}{\rm (Schauder estimates)}\label{Heat-Smooth1}
{\rm(i)} Let $f\in\mathcal B_{p,q}^{\alpha}(\mathbb T^d)$ for some $\alpha\in\mathbb R,~p,q\in[1,\infty]$.
Then for every $\delta>0$,  uniformly over $t>0$
\begin{equation}\label{Heat-Smooth}
\|e^{tA} f\|_{\mathcal B_{p,q}^{\alpha+\delta}(\mathbb T^d)}\lesssim t^{-\frac{\delta}{2}}\|f\|_{\mathcal B_{p,q}^{\alpha}(\mathbb T^d)}.
\end{equation}
{\rm (ii)} Let $\alpha\leq \beta\in\mathbb R$ be such that $\beta-\alpha\leq2$, $f\in \mathcal B_{p,q}^{\beta}(\mathbb T^d)$ and $p,q\in[1,\infty]$. Then uniformly over $t>0$
\begin{equation}\label{Heat-time-reg}
\|({\rm I}-e^{tA}) f\|_{\mathcal B_{p,q}^{\alpha}(\mathbb T^d)}\lesssim t^{\frac{\beta-\alpha}{2}}\|f\|_{\mathcal B_{p,q}^{\beta}(\mathbb T^d)}.
\end{equation}
\end{lemma}

The following multiplicative inequalities play a central role  later and we  treat separately for  cases of positive and negative regularity (cf.  \cite[Corollaries 3.19,3.21]{MW2010}, \cite[Lemma 2.1]{GIP13}).

\begin{lemma}{\rm (Multiplicative inequalities)}\label{Multi-ineq}
{\rm (i)} Let $\alpha>0$ and $p,p_1,p_2\in[1,\infty]$ be such that $\frac{1}{p}=\frac{1}{p_1}+\frac{1}{p_2}$. Then
\begin{equation}\label{Multi-ineq-1}
\|fg\|_{\mathcal B_{p,q}^{\alpha}(\mathbb T^d)}\lesssim \|f\|_{ \mathcal B_{p_1,q}^{\alpha}(\mathbb T^d)}\|g\|_{\mathcal B_{p_2,q}^{\alpha}(\mathbb T^d)}.
\end{equation}
{\rm (ii)} Let $\beta>0>\alpha$ be such that $\beta+\alpha>0$ and let $p,p_1,p_2\in[1,\infty]$ be such that $\frac{1}{p}=\frac{1}{p_1}+\frac{1}{p_2}$. Then
\begin{equation}\label{Multi-ineq-2}
\|fg\|_{\mathcal B_{p,q}^{\alpha}(\mathbb T^d)} \lesssim
\|f\|_{\mathcal B_{p_1,q}^{\alpha}(\mathbb T^d)}\|g\|_{\mathcal B_{p_2,q}^{\beta}(\mathbb T^d)}.
\end{equation}
\end{lemma}

Throughout the paper we consider the equations on $\mathbb T^2$. For notations' simplicity, for any $\alpha\in\mathbb R$ and $p,q\in[1,\infty)$, let
\begin{equation}\label{nota}
\mathcal B^{\alpha}_{p,q}:=\mathcal B^{\alpha}_{p,q}(\mathbb T^2),~~\mathcal B^{\alpha}_p:=\mathcal B^{\alpha}_{p,\infty}(\mathbb T^2),~~\mathcal C^{\alpha}:=\mathcal B_{\infty,\infty}^{\alpha}(\mathbb T^2)
\end{equation}
and hence we denote their norms by $\|\cdot\|_{\mathcal B^{\alpha}_{p,q}},\|\cdot\|_{\mathcal B^{\alpha}_p}$ and $\|\cdot\|_{{\alpha}}$, respectively.
\begin{lemma}{\rm \cite[Proposition A.11]{TW2018}}\label{est-P-PN-P}
Let $P_N,N\geq 1$ be defined in (\ref{pro-ope}). Then for every $\alpha\in\mathbb R$, $p,q\in[1,\infty]$ and $\lambda>0$
\begin{equation}\label{est-P-PN}
\|P_Nf-f\|_{\mathcal B^{\alpha}_{p,q}}\lesssim\frac{(\log N)^2}{N^{\lambda}}\|f\|_{\mathcal B^{\alpha+\lambda}_{p,q}},
\end{equation}
\begin{equation}\label{est-PN}
\|P_Nf\|_{\mathcal B^{\alpha}_{p,q}}\lesssim \|f\|_{\mathcal B^{\alpha+\lambda}_{p,q}}.~~~~~~
\end{equation}
\end{lemma}

\end{section}

\section{Stochastic heat equation}\label{SEC3}
In this section we prove the convergence rates for  stochastic  heat equation \eqref{ini1-Eq}.

\subsection{Wick powers}\label{set-Wick powers}
Now we follow the idea from \cite{MW2010}  to define the Wick powers of the solutions
to stochastic heat equations \eqref{ini1-Eq} in the paths space. Let $(\Omega,\mathcal F,\mathbb P)$ be a probability space and $\xi$ is space-time white noise on $\mathbb R\times \T$. Set
\begin{equation*}\label{fil}
\tilde{\mathcal F}_t:=\sigma\big(\{\xi(\phi):~\phi|_{(t,\infty)\times\T}\equiv0,~\phi\in L^2((-\infty,\infty)\times \T)\}\big)
\end{equation*}
for $t>-\infty$ and denote by $(\mathcal F_t)_{t>-\infty}$ the usual augmentation (see \cite[Chapter 1.4]{RY1999}) of the filtration $(\tilde{\mathcal F}_t)_{t>-\infty}$.
For $n=1,2,3$,  consider the multiple stochastic integral given by
\begin{equation}\label{HEAT-MILD-s-23}
Z_{-\infty,t}^{:n:}(\phi):=\int_{\{(-\infty,t]\times\T\}^n}\langle\phi,\Pi_{i=1}^n H(t-s_i,x_i-\cdot)\rangle\xi(\otimes_{i=1}^n \d s_i,\otimes_{i=1}^n \d x_i)
\end{equation}
for every $t>-\infty$ and $\phi\in C^{\infty}(\T)$. $Z_{-\infty,\cdot}^{:1:}$ is also denoted as $Z_{-\infty,\cdot}$. Here $H(r,\cdot),~r\neq0$, stands for the periodic heat kernel associated to the generator $A=\Delta-{\rm I}$ on $\T$ given by
\begin{equation}\label{ker-perid}
H(r,x):=\sum_{m\in\Z}e^{-rI_m}e_m(x),~x\in \T,~r\in\mathbb R \setminus \{0\},
\end{equation}
with $I_m:=1+4\pi^2|m|^2$ and $e_m=e^{\iota2\pi m\cdot}$ for $m\in\Z$.  Let $S(t):=e^{-tA}$ denote the semigroup associated to  $A$ in $L^2(\T)$.
$Z^{:n:}_{-\infty,\cdot}$ is called the $n$-th Wick power of $Z_{-\infty,\cdot}$. In particular, using Duhamel's principle (c.f. \cite[Section 2.3]{Ev10}), we have that
\[
   Z_{t}:=Z_{-\infty,t}-S(t)Z_{-\infty,0},~t\geq 0,
\]
solves the stochastic heat equation with zero initial condition, i.e.
\begin{equation}\label{HEAT-s}
\left\{
 \begin{aligned}
\t Z&=AZ+\xi,~~\text{on~}(0,\infty)\times\T,\\
Z({0})&=0~~\text{on~}\T.
 \end{aligned}
 \right.
\end{equation}
Following the technique in \cite[Section 2.1]{TW20182}, we set for $n=1, 2,3$
\begin{equation}\label{Z230twick}
\begin{aligned}
    Z^{:n:}_{t}:=\sum_{k=0}^n\binom{n}{k}(-1)^k\Big(S(t)Z_{-\infty,0}\Big)^kZ_{-\infty,t}^{:n-k:},
\end{aligned}
\end{equation}
by letting $ z^{:1:}=z$ and $z^{:0:}=1$.

\subsection{Finite dimensional approximations}\label{Finite-App-section}
Let $ Z_{-\infty,t}^{:n:},Z_{t}^{:n:}, t\geq 0$  be defined in the above section. For $N\geq 1$ and let $P_N$ be given in \eqref{pro-ope}, we define  finite dimensional approximations of $Z_{-\infty,t}$  as its Galerkin projection $Z^N_{-\infty,t}:=P_NZ_{-\infty,t}$ and define  finite dimensional approximations of $Z_{-\infty,t}^{:n:}$, $n=2,3$ by renormalization as
\begin{equation}\label{Zst2-3-APP}
   ( Z^N_{-\infty,t})^{:2:}:= (Z^N_{-\infty,t})^2-  \mathfrak R^N,~
   ( Z^N_{-\infty,t})^{:3:}:= (Z^N_{-\infty,t})^3-  3\mathfrak R^NZ^N_{-\infty,t},
\end{equation}
with the renormalization  constants
\begin{equation*}\label{renor-cosnt}
    \mathfrak R^N:=\|1_{[0,\infty)}H_N\|_{L^2(\mathbb R\times \T)}^2,
\end{equation*}
where $H_N:=P_NH$ and $H$  is given in \eqref{ker-perid}. Comparing with \eqref{Z230twick}, we define finite dimensional approximations of $Z_t^{:n:}$ as for $N\geq 1$, $t\geq 0$ and $n=1,2,3$
\begin{equation}\label{ZN1230t-wick}
\begin{aligned}
    (Z_{t}^N)^{:n:}:=\sum_{k=0}^n\binom{n}{k}(-1)^k\Big(S(t)Z_{-\infty,0}^N\Big)^k(Z_{-\infty,t}^N)^{:n-k:}
\end{aligned}
\end{equation}
by similarly letting $ z^{:1:}=z$ and $z^{:0:}=1$.  Then $Z_{t}^N$ solve approximating equations \eqref{ini1-Eq_N} with  initial value zero.

\subsection{Main results of  stochastic heat equation and its Galerkin approximations} \label{SUB3-3}
We know that $Z_t$ and $Z^N_t,N\geq 1$ are the solutions to  stochastic heat equation \eqref{ini1-Eq} and its Galerkin approximations \eqref{ini1-Eq_N} with  initial value zero, respectively.
In this subsection, we first prove the convergence rate for  stochastic heat equation  with  initial value zero. Furthermore, we discuss the convergence rate for linear equation with  initial value $X_0$.

To begin with, we recall the following   Kolmogorov-type result.
\begin{lemma}{\rm \cite[Lemma 5.2]{MW2010}} \label{exist-Z-cont}
Let $(t,\phi)\rightarrow Z(t,\phi)$ be a map from $(0,\infty)\times L^2(\T)\rightarrow L^2(\Omega,\mathcal F,\mathbb P)$ which is linear and continuous in $\phi$. Assume that for some $p>1$, $\alpha\in\mathbb R$ and $\nu>\frac{1}{p}$ and all $T>0$, there exists a function $K_T\in L^{\infty}(\T)$ such that for  all $\kappa\geq-1,~x\in\T$ and $s,t\in[0,T]$
\begin{equation}\label{Z-exist-1}\begin{aligned}
\mathbb E|Z(t,\eta_\kappa(\cdot-x))|^p&\leq K_T(x)^p2^{-\kappa\alpha p},\\
\mathbb E|Z(t,\eta_\kappa(\cdot-x))-Z(s,\eta_\kappa(\cdot-x))|^p&\leq K_T(x)^p2^{-\kappa(\alpha-\nu) p}|t-s|^{\nu p}
\end{aligned}
\end{equation}
with $\eta_\kappa,\kappa\geq -1$ defined in \eqref{LPBs2}. Then there exists a random distribution $\tilde{Z}$ which is $C([0,\infty);\mathcal B^{\alpha'}_{p,p})$-valued for any $\alpha'<\alpha-\nu$ and satisfies that for all $t>0$ and $\phi\in\mathcal S(\T)$
\begin{equation*}
\label{cont-Z-EQ}
Z(t,\phi)=\langle\tilde{Z}(t),\phi\rangle ~~\text{almost surely.}
\end{equation*}
Furthermore, for every $T>0$, there exists a constant $C(T,\alpha,\alpha',p)>0$ such that
\begin{equation}
\label{Lp-Z-EQ}
\mathbb E \sup_{0\leq t\leq T}\|\tilde{Z}(t,\cdot)\|_{\mathcal B_{p,p}^{\alpha'}}^p\leq C(T,\alpha,\alpha',p)  \|K_T\|_{L^p}^p.
\end{equation}
\end{lemma}

We use Lemma \ref{exist-Z-cont} to obtain  the regularity properties of   $Z_t^{:n:},(Z^N_t)^{:n:}$.

  We remark that properties for  stochastic heat equation have been widely discussed in the literature (cf. \cite[Propositions 2.2,2.3]{TW20182},  \cite[Proposition 7.4]{TW2018}, \cite[Lemma 3.4]{RZZ2017}). Below we obtain uniform bounds  for the approximating equations.

\begin{lemma}\label{modefi-Z}
Let $p>1$. Then for each $\alpha\in(0,1)$, $n=1,2,3$ and $N\geq 1$, the processes $Z_{-\infty,\cdot}^{:n:}$ and $(Z^N_{-\infty,\cdot})^{:n:}$ defined in \eqref{HEAT-MILD-s-23} and \eqref{Zst2-3-APP}  belong to  $ C([0,T]; \C^{-\alpha})$ $\mathbb P$-a.s. Moreover, we have
\begin{equation}\label{Z^nt-lp}
   \E \sup_{0\leq t\leq T}\|Z^{:n:}_{-\infty,t}\|_{ {-\alpha}}^p<\infty,
\end{equation}
\begin{equation}\label{ZN^nt-lp}
 \sup_{N\geq1}   \E \sup_{0\leq t\leq T}\|(Z^N_{-\infty,t})^{:n:}\|_{ {-\alpha}}^p<\infty.
\end{equation}
\end{lemma}
\begin{proof}
\eqref{Z^nt-lp}  has been obtained in \cite[Theorem 2.1]{TW20182} and we only prove \eqref{ZN^nt-lp}.
By \eqref{LPBs2} we set $\phi(\cdot)=\eta_j(x-\cdot)$ in \eqref{HEAT-MILD-s-23} and obtain that
\[
 \Delta_{j}Z^N_{-\infty,t}(x)=Z^N_{-\infty,t}(\eta_j(x-\cdot)),~~x\in\T,t\in[0,T],~~j\geq-1.
\]
Using the results in  the proofs of {\cite[Theorem 2.1, Proposition 2.3]{TW20182}}, which are based on \eqref{iTO-EQ} and the semigroup property of $H(t,x)$ defined in \eqref{ker-perid}, we obtain that for $x_1,x_2\in\T$ and $s,t\in[0,T]$
\begin{equation}\label{KAPPAzn1}
\begin{aligned}
\E\Delta_{j}(Z^N_{-\infty,s})^{:n:}&(x_1)\Delta_{j}(Z^N_{-\infty,t})^{:n:}(x_2)\backsimeq\\
 & n! \sum_{\substack{m_1\in\mathcal A_{2^{j}},\\ m_1\in\mathbb Z^2}} \sum_{\substack{|m_i|\leq N,\\ m_i\in\mathbb Z^2,i=2,\ldots,n}} \prod^n_{i=1}
  \frac {e^{-I_{m_i-m_{i-1}}|s-t|}}{2I_{m_i-m_{i-1}}}  e_{m_1}(x_1-x_2),
\end{aligned}
\end{equation}
with  the convention that $m_0 = 0$ for $I_m=1+4\pi^2|m|^2$, $e_m=e^{\iota2\pi m\cdot}$,  $m\in\Z$. Let $K^\gamma(m)=\frac 1 {(1+|m|^2)^{1-\gamma}}$ for $\gamma\in[0,1)$. If we choose $x_1=x_2=x$ and  $s=t$, then we get an estimate of the form
\[
   \E\left.|\Delta_{j}(Z^N_{-\infty,t})^{:n:}(x)\right.|^2\lesssim
   \sum_{m\in\mathcal A_{2^{j}},m\in\mathbb Z^2} K^0 \star^{n}_{\leq N} K^0 (m),
\]
with $K^\gamma \star^{n}_{\leq N} K^\gamma$,$\gamma\in[0,1)$ defined in \eqref{<Nn}, while for $s\neq t$ and every $\gamma\in(0,1)$
\[
   \E\left.|\Delta_{j}(Z^N_{-\infty,s})^{:n:}(x)
   -\Delta_{j}(Z^N_{-\infty,t})^{:n:}(x)\right.|^2
   \lesssim  |s-t|^{n\gamma}
   \sum_{m\in\mathcal A_{2^{j}},m\in\mathbb Z^2} K^\gamma \star^{n}_{\leq N} K^\gamma (m).
\]
 Then using the estimates in Lemma  \ref{KENEL-ES} we have for  any $\lambda>0$ and $\gamma\in(0,\frac 1 n)$
\begin{equation*}\label{5THZN-ES}\begin{aligned}
   \E\big|\Delta_{j}(Z^N_{-\infty,t})^{:n:}(x)\big|^2\lesssim&
   \sum_{m\in\mathcal A_{2^{j}},m\in\mathbb Z^2} \frac{1}{(1 + |m|^2)^{1-\lambda}},\\
\E \left|\Delta_{j}(Z^N_{-\infty,s})^{:n:}(x)-\Delta_{j}(Z^N_{-\infty,t})^{:n:}(x)\right|^2
 \lesssim&
   \sum_{m\in\mathcal A_{2^{j}},m\in\mathbb Z^2} \frac{ |s-t|^{n\gamma}}{(1 + |m|^2)^{1-n\gamma}}
\end{aligned}
\end{equation*}
 uniformly for  $x\in\T$, $s,t\in[0,T]$, $j\geq -1$, $N\geq 1$. Considering that $|m|\lesssim 2^{j}$ for $m\in\mathcal A_{2^j}$, we further have estimates  for the above $\gamma$ and $\lambda>\frac {n\gamma}{2}$
 \begin{equation*}\label{5THZN-ES}\begin{aligned}
   \E\big|\Delta_{j}(Z^N_{-\infty,t})^{:n:}(x)\big|^2\lesssim&
   2^{2j\lambda},\\
\E \left|\Delta_{j}(Z^N_{-\infty,s})^{:n:}(x)-\Delta_{j}(Z^N_{-\infty,t})^{:n:}(x)\right|^2
 \lesssim&  |s-t|^{n\gamma}2^{2j n\gamma }  \lesssim  |s-t|^{n\gamma}2^{2j(\lambda+ \frac {n\gamma}{2}) }.
\end{aligned}
\end{equation*}
Let $p\geq 2$, then by \eqref{iTO-iEQ} and the above estimates we have that for  $0<\gamma<\frac 1 n $ and $\lambda>\frac{n\gamma}{2}$,
\begin{equation}\label{time-reg-Z}
\E \left|\Delta_{j}(Z^N_{-\infty,t})^{:n:}(x)\right|^p
\lesssim 2^{j \lambda p},~
\end{equation}
\begin{equation}\label{time-reg-Z2}
\E \left|\Delta_{j}(Z^N_{-\infty,s})^{:n:}(x)-\Delta_{j}(Z^N_{-\infty,t})^{:n:}(x)\right|^p
\lesssim |s-t|^{\frac{n\gamma p}{2}} 2^{j(\lambda+\frac{n\gamma}{2})p},
\end{equation}
where the constants we omit are independent of $N$. Hence \eqref{Z-exist-1} holds by Lemma \ref{exist-Z-cont} and choosing  $\nu=\frac{n\gamma}{2}$ and $\alpha=-\lambda$ for $\lambda>\frac{n\gamma}{2}$,  and the embedding $\mathcal B_{p,p}^{-\alpha+\frac{2}{p}}\hookrightarrow\C^{-\alpha}$ for $\alpha>\frac{2}{p}$. Then we have modifications still denoted by
 ${Z}_{-\infty,\cdot}^{:n:}$,$({Z}^N_{-\infty,\cdot})^{:n:}$ in $C([0,T];\mathcal C^{-\alpha})$  for any $\alpha>{n\gamma}+\frac{2}{p}$.
Moreover, we conclude that ${Z}_{-\infty,\cdot}^{:n:}$,$({Z}^N_{-\infty,\cdot})^{:n:}\in C([0,T];\mathcal C^{-\alpha})$, $\mathbb P$-a.s. for any $\alpha>0$ since by the arbitrariness of $\gamma,p$  we can choose $\gamma$ small enough and $p$ sufficiently large. By \eqref{Lp-Z-EQ}, \eqref{time-reg-Z}, \eqref{time-reg-Z2} and  Cauchy-Schwarz inequality, \eqref{ZN^nt-lp} holds for all $p>1$.

\end{proof}

\begin{lemma}\label{ZtN-ZsN0}
Let $\alpha\in (0,1)$, $n=1,2,3$  and $p>1$. Then for every $\alpha'>0$
\begin{equation}\label{est-ZN-HOLD}
     \sup_{N\geq 1}\E \sup_{0\leq t\leq T}t^{(n-1){\alpha}'p}\|(Z^N_{t})^{:n:}\|_{{-\alpha}}^p<\infty,
 \end{equation}
 \begin{equation}\label{est-Z-HOLD}
     \E \sup_{0\leq t\leq T}t^{(n-1){\alpha}'p}\|Z_{t}^{:n:}\|_{{-\alpha}}^p<\infty.
 \end{equation}
\end{lemma}
\begin{proof}
 Let   $0<\bar\alpha<\alpha\wedge \alpha'$ be fixed  and   $\epsilon>0$ small enough. By Lemmas  \ref{modefi-Z}, \ref{Heat-Smooth1} and \ref{Multi-ineq} we have
 $(Z_{-\infty,t}^N)^{:n:}\in\C^{-\bar\alpha}$, $S(t)Z_{-\infty,t}^N\in\C^{\bar \alpha+\epsilon}$  for every $t\in[0,T]$ and
\[ t^{{\bar\alpha+\frac \epsilon {2}}}\|S(t)Z_{-\infty,0}^N\|_{\bar \alpha+\epsilon}\lesssim  \|Z_{-\infty,0}^N\|_{-\bar \alpha},
\]
\[
 t^{{(\bar\alpha+\frac \epsilon {2})(n-1)}} \|(S(t)Z_{-\infty,0}^N)^n\|_{-\bar\alpha}\lesssim  \|Z_{-\infty,0}^N\|_{-\bar \alpha}^n.
\]
Then using Lemmas \ref{Multi-ineq} and \ref{V-SUB-H} in \eqref{ZN1230t-wick} we have
\begin{equation*}
    \|(Z_{t}^N)^{:n:}\|_{-\alpha}\lesssim \sum_{k=0}^{n-1}
    \|S(t)Z_{-\infty,0}^N\|_{\bar\alpha+\epsilon}^k\|(Z_{-\infty,t}^N)^{:n-k:}\|_{-\bar\alpha}
     +\|(S(t)Z_{-\infty,0}^N)^n\|_{-\bar\alpha},
\end{equation*}
and \eqref{est-ZN-HOLD} follows  by Lemma \ref{modefi-Z} and Cauchy-Schwarz's inequality.

 \eqref{est-Z-HOLD} can be similarly obtained and we omit the proof. See also \cite[Proposition 2.2]{TW20182} for the details.
\end{proof}

Using the moment bounds obtained in Lemma \ref{ZtN-ZsN0}, we continue to prove the convergence  rate for linear equation \eqref{HEAT-s}  with initial value  zero  in Lemma \ref{Zt-ZtN}. Similar convergence results can, e.g., be found in \cite[Propositions 2.2,2.3]{TW20182},  \cite[Proposition 7.4]{TW2018}. The main difference is that we obtain  the convergence rates.

\begin{lemma}\label{Zt-ZtN}
Let $\alpha\in(0,1)$,   $n=1,2,3$ and $p>1$. Then
\begin{equation}\label{Z-ZNLp-est}
 ~~~~ \E \sup_{0\leq t\leq T}
    \|Z^{:n:}_{-\infty,t}-(Z^N_{-\infty,t})^{:n:}\|_{ {-\alpha}}^p
   \lesssim (1+N^2)^{-\frac{p \alpha^-}{2}},
\end{equation}
and for every $\beta>\alpha$
\begin{equation}\label{ZTN-ZN-n}
  \E \sup_{0\leq t\leq T}  t^{\frac{{(n-1)}(\alpha+\beta)p}{2}}
    \|Z^{:n:}_{t}-(Z^N_{t})^{:n:}\|_{ {-\alpha}}^p \lesssim (1+N^2)^{-\frac{p \alpha^-}{2}},
\end{equation}
where  $\alpha^-$  denotes $\alpha-\delta$ for every $\delta>0$.
\end{lemma}

\begin{proof}

Let $N\geq 1$,  similar as  in \eqref{KAPPAzn1} we have that for $x_1,x_2\in\T$ and $s,t\in[0,T]$
\begin{equation}\label{KAPPAz}
\begin{aligned}
\E\Delta_{j}Z_{-\infty,s}^{:n:}&(x_1)\Delta_{j}Z_{-\infty,t}^{:n:}(x_2)\backsimeq\\
 & n! \sum_{m_1\in\mathcal A_{2^{j}},m_1\in\mathbb Z^2} \sum_{ m_i\in\Z, i=2,\ldots,n} \prod^n_{i=1}
  \frac {e^{-I_{m_i-m_{i-1}}|s-t|}}{2I_{m_i-m_{i-1}}}  e_{m_1}(x_1-x_2),
\end{aligned}
\end{equation}
\begin{equation}\label{notd2}
\begin{aligned}
\E\Delta_{j}Z_{-\infty,s}^{:n:}&(x_1)\Delta_{j}(Z^N_{-\infty,t})^{:n:}(x_2)\backsimeq\\
 & n! \sum_{\substack{m_1\in\mathcal A_{2^{j}},\\m_1\in\Z}} \sum_{\substack{|m_i|\leq N,\\m_i\in \Z,i=2,\ldots,n}} \prod^n_{i=1}
  \frac {e^{-I_{m_i-m_{i-1}}|s-t|}}{2I_{m_i-m_{i-1}}}  e_{m_1}(x_1-x_2),~~~~~~
\end{aligned}
\end{equation}
with $m_0:= 0$,  $I_m=1+4\pi^2|m|^2$, $e_m=e^{\iota2\pi m\cdot}$ for $m\in\Z$. Then by letting  $x_1=x_2=x$ and $s=t$ in the above two estimates and in \eqref{KAPPAzn1}  and by equality $(a-b)^2=a^2+b^2-2ab$ we have
\begin{equation*}
\begin{aligned}
   \E\big|\Delta_{j}Z_{-\infty,t}^{:n:}(x)
  & -\Delta_{j}(Z^N_{-\infty,t})^{:n:}(x)\big|^2
   \backsimeq \frac {n!}{2}
   \sum_{m_1\in\mathcal A_{2^{j}},m_1\in\Z}
    \Big\{\sum_{m_i\in\Z ,i=2,\ldots,n} \prod^n_{i=1}
  \frac {1}{1+|m_i-m_{i-1}|^2} \\
  &-  \sum_{|m_i|\leq N,m_i\in\Z ,i=2,\ldots,n} \prod^n_{i=1}
  \frac {1}{1+|m_i-m_{i-1}|^2}  \Big\}
   \lesssim   \sum_{m\in\mathcal A_{2^{j}},m\in\Z} K^0 \star^n_{>N} K^0 (m),
\end{aligned}
\end{equation*}
with $K^0(m):=\frac 1 {1+|m|^2}$ and $K^0 \star^{n}_{> N} K^0$ defined in \eqref{>Nn}, and  by Lemma  \ref{KENEL-ES} we have for any positive $\lambda$ satisfying $\lambda<1-\epsilon$ that
\[\sum_{m\in\mathcal A_{2^{j}},m\in\Z} K^0 \star^n_{>N} K^0 (m)\lesssim
 (1+N^2)^{-\lambda}\sum_{m\in\mathcal A_{2^{j}},m\in\Z}{(1+|m|^2)^{\epsilon+\lambda-1}}
 \lesssim
 (1+N^2)^{-\lambda}2^{2j{(\epsilon+\lambda)}}.
\]
Then by \eqref{iTO-iEQ}, we have for every $p\geq 2$, $\lambda,\epsilon>0$ such that $\lambda+\epsilon<1$,
\begin{equation*}
\begin{aligned}
\E\left| \Delta_j Z_{-\infty,t}^{:n:}(x)- \Delta_j (Z^N_{-\infty,t})^{:n:}(x)\right|^p\lesssim 2^{j(\lambda+\epsilon) p}{(1+N^2)^{-\frac{\lambda p}{2}}},
\end{aligned}
\end{equation*}
uniformly for  $x\in\T$, $t\in[0,T]$, $j\geq -1$, $N\geq 1$.
Then for any  $\alpha\in (0,1)$, we  choose $p$ sufficiently large, $\epsilon$ sufficiently small  and $\lambda>0$ such that $\lambda+\epsilon<\alpha-2/p$ and thus we obtain
\begin{equation*}
\E \sup_{0\leq t\leq T}
    \|Z^{:n:}_{-\infty,t}-(Z^N_{-\infty,t})^{:n:}
         \|_{\mathcal B^{-\alpha+\frac 2 p}_{p,p}}^p
   \lesssim (1+N^2)^{-\frac{ \lambda  p}{2}}.
\end{equation*}
Finally, using the embedding $\mathcal B_{p,p}^{-\alpha+{2}/{p}}\hookrightarrow\C^{-\alpha}$ for any $\alpha>\frac{2}{p}$ and letting $\lambda$ close to $\alpha$ and $\epsilon$ close to 0, we have
\begin{equation*}
\E \sup_{0\leq t\leq T}
    \|Z^{:n:}_{-\infty,t}-(Z^N_{-\infty,t})^{:n:}\|_{-\alpha}^p
   \lesssim
 (1+N^2)^{-\frac{  p \alpha^-  }{2}}
\end{equation*}
for $p$  large enough. Then by the Cauchy-Schwarz inequality, \eqref{Z-ZNLp-est} holds for all $p>1$.

We continue to prove \eqref{ZTN-ZN-n}.  For every  $\beta>\alpha$, proceeding as in  Lemma \ref{ZtN-ZsN0}, we obtain $S(t)Z_{-\infty,0}^N\in\C^{ \beta}$ and
\[ t^{\frac{\alpha+\beta}{2}}\|S(t)Z_{-\infty,0}^N\|_{\beta}\lesssim  \|Z_{-\infty,0}^N\|_{-\alpha},
\]
\[
   t^{\frac{n(\alpha+\beta)}{2}}
 \|(S(t)Z_{-\infty,0})^n-(S(t)Z_{-\infty,0}^N)^n\|_{\beta}\lesssim
  \|Z_{-\infty,0}-Z_{-\infty,0}^N\|_{-\alpha}(\|Z_{-\infty,0}\|^{n-1}_{-\alpha}+\|Z_{-\infty,0}^N\|^{n-1}_{-\alpha}),
\]
\[
   t^{\frac{{(n-1)}(\alpha+\beta)}{2}}
 \|(S(t)Z_{-\infty,0})^n-(S(t)Z_{-\infty,0}^N)^n\|_{-\alpha}\lesssim
  \|Z_{-\infty,0}-Z_{-\infty,0}^N\|_{-\alpha}(\|Z_{-\infty,0}\|^{n-1}_{-\alpha}+\|Z_{-\infty,0}^N\|^{n-1}_{-\alpha}).
\]
Then using Lemma  \ref{Multi-ineq} in \eqref{ZN1230t-wick} and \eqref{Z230twick} we have
\begin{equation*}\label{ZN-Z-wick-0t}
\begin{aligned}
    \|Z^{:n:}_{t}-(Z^N_{t})^{:n:}\|_{-\alpha}\lesssim&\sum_{k=0}^{n-1}\big[
    \|(S(t)Z_{-\infty,0}^N)\|_{\beta}^k\cdot\|Z_{-\infty,t}^{:n-k:}-(Z_{-\infty,t}^N)^{:n-k:}\|_{-\alpha}\\
        &+ \|(S(t)Z_{-\infty,0})^k-(S(t)Z_{-\infty,0}^N)^k\|_{\beta} \cdot  \|Z_{-\infty,t}^{:n-k:}\|_{-\alpha}\big]\\
 &+ \|(S(t)Z_{-\infty,0})^n-(S(t)Z_{-\infty,0}^N)^n \|_{-\alpha},
\end{aligned}
\end{equation*}
and \eqref{ZTN-ZN-n} follows by the Cauchy-Schwarz inequality, \eqref{ZN^nt-lp} and \eqref{Z-ZNLp-est}.

\end{proof}

\bigskip
Now following the techniques in \cite{MW2010,RZZ2017}, we combine the initial value part with the Wick
powers. Let $X_0\in\C^{-\alpha},\alpha\in(0,1)$, we set $V_t := S(t)X_0$, $V^N_t := P_NV_t$ and
\[
 \bar Z_t:=Z_t+V_t ,~  \bar Z^N_t:=Z^N_t+V_t^N,
\]
\begin{equation}\label{ZN1230t-wick-ini}
 \bar Z_t^{:n:}:=\sum_{k=0}^n\binom{n}{k}  V_t^{n-k} Z_t^{:k:},~ (\bar Z_t^N)^{:n:}:=\sum_{k=0}^n\binom{n}{k}  (V_t^N)^{n-k} (Z_t^N)^{:k:}
\end{equation}
for  $n=1,2,3$ with $\bar Z_t^{:1:}=\bar Z_t$. Then we know that $\bar Z_t$, $\bar Z_t^N$ are  the solutions to \eqref{ini1-Eq} and \eqref{ini1-Eq_N} with initial values $X_0$ and $P_NX_0$, respectively. By Lemma \ref{Heat-Smooth1} we have $V\in C([0,T];\C^{-\alpha})$ and $V\in C([0,T];\C^{\beta})$ for every $\beta>-\alpha$ with the norm $\sup_{t\in  [0,T]}t^{\frac{\alpha+\beta}{2}}\|\cdot\|_{{\beta}}$. Moreover, together with \eqref{Heat-Smooth} we have
\begin{equation}  \label{V-ini}
\sup_{0\leq t\leq T}t^{\frac{\alpha+\beta}{2}}\|V_t\|_{{\beta}}\lesssim \|X_0\|_{-\alpha},~
\sup_{0\leq t\leq T}t^{\frac{\alpha+\beta+\kappa}{2}}\|V_t^N\|_{{\beta}}\lesssim \|X_0\|_{-\alpha},~
\end{equation}
for  $\beta\geq-\alpha$ and  $\kappa>0$. Then we extend the results in Lemmas \ref{modefi-Z}-\ref{Zt-ZtN} to the solution to stochastic heat equation \eqref{ini1-Eq} with   initial value $X_0$.

\begin{theorem}
\label{Z-INI}
Let $\alpha\in(0,1)$, $X_0\in\C^{-\alpha}$, $n=1,2,3$, $p>1$. Let $\bar Z^{:n:}$ and ${(\bar Z^N)}^{:n:}$ be defined  in \eqref{ZN1230t-wick-ini}. Then $({\bar Z^N})^{:n:}$ converges to ${\bar Z}^{:n:}$ in $L^p(C([0,T];\C^{-\alpha}))$  such that for every $\beta>\alpha$ and $\kappa>0$
\begin{equation}\label{est-ZN-HOLD-INI1}
     \sup_{N\geq 1}\E \sup_{0\leq t\leq T}t^{\frac{(\alpha+\beta){(n-1)}+n\kappa}{2} p}\|(\bar Z^N_{t})^{:n:}\|_{-\alpha}^p<\infty,
\end{equation}
\begin{equation}\label{est-ZN-HOLD-INI2}
     \E \sup_{0\leq t\leq T}t^{\frac{(\alpha+\beta){(n-1)}}{2} p}\|\bar Z_{t}^{:n:}\|_{-\alpha}^p<\infty,
 \end{equation}
 \begin{equation}\label{ZTN-ZN-n-ini}
  \E \sup_{0\leq t\leq T}  t^{\frac{(\alpha+\beta){(n-1)}+n\kappa}{2} p}
    \|\bar Z^{:n:}_{t}-(\bar Z^N_{t})^{:n:}\|_{ {-\alpha}}^p \lesssim \frac{(\log N)^{2p}}{N^{\kappa p}}+(1+N^2)^{-\frac{ p\alpha^-}{2}},
\end{equation}
where  $\alpha^-$  denotes $\alpha-\delta$ for every $\delta>0$.
\end{theorem}

\begin{proof}
Using Lemma   \ref{Heat-Smooth1}  in \eqref{ZN1230t-wick-ini},  we have for $\beta>\alpha$,
\[
\|{(\bar Z_t^N)}^{:n:}\|_{-\alpha}
\lesssim\sum^{n-1}_{k=0}\|V_t^N\|^k_{{\beta}}\|{( Z_t^N)}^{:n-k:}\|_{-\alpha}+\|V_t^N\|_{{-\alpha}}\|V_t^N\|^{n-1}_{\beta},
\]
which together with \eqref{V-ini} and \eqref{est-ZN-HOLD}, implies \eqref{est-ZN-HOLD-INI1} easily. Similarly we obtain \eqref{est-ZN-HOLD-INI2}.

Combining Lemma  \ref{Multi-ineq} with  \eqref{V-ini}, we have for every $\beta>\alpha$, $\kappa>0$
\begin{equation}\label{ZTN-ZN-n-v}
\begin{aligned}
 t^{\frac{(\alpha+\beta+\kappa)n}{2}}  \|(V_t^N)^n-V^n_t\|_{{\beta}} &\lesssim  \frac{(\log N)^2}{N^\kappa}  \|X_0\|^n_{-\alpha},\\
 t^{\frac{(\alpha+\beta)(n-1)+n\kappa}{2}}  \|(V_t^N)^n-V^n_t\|_{{-\alpha}} &\lesssim  \frac{(\log N)^2}{N^\kappa}  \|X_0\|^n_{-\alpha}.
\end{aligned}
\end{equation}
Then finally  in \eqref{ZN1230t-wick-ini} we get
\begin{equation*}\begin{aligned}
\|{(\bar Z_t^N)}^{:n:}-{\bar Z_t}^{:n:} \|_{-\alpha}&\lesssim
   \sum^{n-1}_{k=0}\big[\|V_t^N\|^k_{{\beta}}\|{( Z_t^N)}^{:n-k:}
           -{ Z_t}^{:n-k:}\|_{-\alpha }
   \\& +\|(V_t^N)^k-V_t^k\|_{{\beta}}\|{Z}^{:n-k:}_t\|_{-\alpha}\big]+\|(V_t^N)^n-V_t^n\|_{{-\alpha}}.
\end{aligned}
\end{equation*}
Now \eqref{ZTN-ZN-n-ini} follows by \eqref{est-Z-HOLD}, \eqref{ZTN-ZN-n} and \eqref{ZTN-ZN-n-v}.

\end{proof}

\section{Main results for stochastic Allen-Cahn equations}\label{Pathwise-error}
Now we fix a stochastic
basis $(\Omega,\mathcal F, \{\mathcal{F}_t\}_{t\in[0,\infty)}, \mathbb P)$. Suppose that $\xi$ is  space-time white noise on $\mathbb R^+\times \T$. Let $T>0$ and $X_0\in\C^{-\alpha}$ with $\alpha\in(0,1)$ and let ${\bar Z},{\bar Z}^N$ be given in Section \ref{SEC3}.
Existence and uniqueness of the solutions to \eqref{ini2-Eq} have been obtained in \cite[Theorem 3.10]{RZZ2017} and  \cite[Theorem 6.2]{MW2010}.
By \cite[Theorem 5.1]{LR13} we obtain  existence and uniqueness of the solutions to \eqref{ini2-Eq_N}. The results are concluded  in the theorem below.

\begin{theorem}     \label{est-Y-HOLD-Th}
Let  $\alpha, \beta>0$ with $\alpha<\beta<\frac 2 3-\alpha$ and  $p> 1$.
Then there exist unique mild solutions $Y$ and $Y^{N}$  on $[0,T]$  to equations \eqref{ini2-Eq} and  \eqref{ini2-Eq_N}, respectively. Here, $Y,Y^{N}\in C([0,T];\C^\beta)$, $Y^{N}\in C([0,T];\C^\infty)$   such  that  for any $\frac {\alpha+\beta} 2<\gamma'<\frac 1 3$
 \begin{equation}\label{est-Y-EHOLD}
\E \sup_{0\leq t\leq T}t^{\gamma' p}\|Y_{t}^{N}\|^p_{\beta}<\infty,~~
\E \sup_{0\leq t\leq T}t^{\gamma' p}\|Y_{t}\|^p_{\beta}<\infty.
 \end{equation}
\end{theorem}
\begin{remark}\label{remgam}
 We  notice that if \eqref{est-Y-EHOLD} holds for $\gamma'<1 /3$, it also holds for $\gamma'\geq 1/3$ since $t^a\lesssim t^b$ for uniform $t\leq T$ with $a\geq b$.
\end{remark}

\bigskip

We conclude that under the setting of Theorem \ref{est-Y-HOLD-Th},
\begin{equation} \label{X-FUNC=y+z}
X=Y+\bar Z,~~X^N=\bar Z^N+Y^N
\end{equation}
are solutions to \eqref{initial-Eq} and \eqref{app-Eq} on $[0,T]$ with initial values $X_0$ and $P_NX_0$, respectively. Since  Theorem \ref{est-Y-HOLD-Th} holds for any $\alpha,\beta,\gamma'>0$ with $\alpha<\beta$, $\frac {\alpha+\beta} 2<\gamma'<\frac 1 3$, then we let $\beta$ close to $\alpha$  and together with  Theorem \ref{Z-INI}, we have for every $\alpha\in(0,1/3)$, $\gamma'>\alpha$ and $p>1$
\begin{equation} \label{X=Y+Z-EST}
\E \sup_{0\leq t\leq T}t^{\gamma' p}\|X_{t}^{N}\|^p_{-\alpha}<\infty,~~
\E \sup_{0\leq t\leq T}t^{\gamma' p}\|X_{t}\|^p_{-\alpha}<\infty.
\end{equation}

\subsection{Pathwise error estimates for stochastic Allen-Cahn equations}\label{SEC4-1}

In the following we fix $\alpha, \beta,p, \gamma',\kappa$ satisfying the following condition:

 $\alpha, \beta>0$ with $\alpha<\beta<\frac 2 3-\alpha$ and $p>1$,
  $\gamma'>0$ with $\frac{2\alpha+\beta}{2}<\gamma'<\frac 1 3$,
and  $\kappa>0$  sufficiently small.

\subsubsection* {Stopping times} \label{S-tim}

 Following the notations in the above section, for some fixed $M>0$ sufficiently large, we define stopping times
 \begin{equation}\label{stp1}
 \tau^M:=\inf\{t\leq T:t^{\gamma' }\|Y_t\|_\beta\geq M\},~
\end{equation}
and
to approximate $Y$ with $Y^N$, we define
 \begin{equation}
 \begin{aligned}\label{stp2}
\sigma_N:&=\inf\{t\leq T:t^{\gamma' }\|Y^N_t-Y_t\|_{\beta}>1\},\\
 \nu_N^{M,\epsilon}:&=\inf\{t\leq T: \sup_{n=1,2,3} t^{\frac{(\alpha+\beta){(n-1)}}{2}}\|\bar Z_t^{:n:}\|_{ {\alpha}}\geq M, \\
    &~~~~~~~~~~~~~~~~~\sup_{n=1,2,3} t^{\frac{(\alpha+\beta){(n-1)+n\kappa}}{2}}\|\bar Z_t^{:n:}-
 (\bar Z^N_t)^{:n:}\|_{ {\alpha}}>\epsilon\},
\end{aligned}
\end{equation}
with $\epsilon>0$ arbitrarily small. For notation's simplicity, we  set
\[
\|\cdot\|_{\mathcal M_\sigma }:=\sup_{t\leq \sigma}t^{\gamma'}\|\cdot\|_\beta
\]
for any stopping time $\sigma$ and
$\| v\|_{ \bar{\mathcal{L}}}:=\sup_{n=1,2,3} \sup_{t\leq T}t^{\frac{(\alpha+\beta){(n-1)+n\kappa}}{2}}\| v_t^{:n:}\|_{ {-\alpha}}.$
 Then by \eqref{mild-2} and \eqref{mild-2-N}
\begin{equation*}\begin{aligned}
   Y_t-Y_t^N=&\int_0^t ({\rm I}-P_N)e^{(t-s)A}
   {\Psi}(Y_s,\underline{\bar Z}_s)ds
+\int_0^t P_Ne^{(t-s)A}\Big\{{\Psi}(Y_s,\underline{\bar Z}_s)-
  {\Psi}(Y_s^N,\underline{\bar Z}_s^N)\Big\}ds.
\end{aligned}\end{equation*}
In the following  we use the decomposition
\begin{equation} \label{Nota-F-De}
 \Psi(u,\underline { z})=F(u)+\widetilde{\Psi}(u,\underline{ z}),
\end{equation}
with $u\in \C^\beta, \underline{ z}=( z,z^{:2:}, z^{:3:})$ that
\begin{equation*} \label{Nota-F}
    F(u):=\sum_{i=0}^3a_iu^i ,~
     \widetilde{\Psi}(u,\underline {\bar z}):=\sum_{i=1}^3a_i\bar z^{:i:}+3a_3(u^2\bar z+u\bar z^{:2:})+2a_2u\bar z+\bar z.
\end{equation*}
Under the assumption that $0<\alpha<\beta,\alpha+\beta<2$ and applying Lemma  \ref{Multi-ineq} and Young's inequality, we easily have that for any $u\in\C^\beta,\underline{ z}=(z,z^{:2:},z^{:3:})$ with $z^{:n:}\in\C^{-\alpha}$ ${F}(u)\in\C^\beta, \widetilde{\Psi}(u,\underline z)\in\C^{-\alpha}$ and
\begin{equation}\begin{aligned}\label{F-PSI}
\|{F}(u)\|_{{\beta}}&\lesssim 1+\|u\|^3_{{\beta}},\\
 \|\widetilde{\Psi}(u,\underline{z})\|_{-\alpha}&\lesssim 1+\|u\|^2_{\beta}\| z\|_{-\alpha}+\|u\|_{\beta}\|z^{:2:}\|_{-\alpha}+\| z^{:3:}\|_{-\alpha},
\end{aligned}
\end{equation}
 and that for $v\in\C^\beta, \underline{w}=(w,w^{:2:},w^{:3:})$ with $w^{:n:}\in\C^{-\alpha}$ we obtain
\begin{equation}\begin{aligned}\label{F-PSI2}
&\|{F}(u)-F(v)\|_{{\beta}}\lesssim \|u-v\|_{{\beta}}( 1+\|u\|^2_{{\beta}}+\|v\|^2_{{\beta}}),\\
 \|\widetilde{\Psi}(u,\underline {z})-&\widetilde{\Psi}(v,\underline{w})\|_{-\alpha}
 \lesssim \big\{ \|u\|_{{\beta}} \| z\|_{-\alpha}+\|v\|_{{\beta}}\| z\|_{-\alpha}+\|z^{:2:}\|_{-\alpha} \big\}\|u-v\|_{{\beta}} \\
  & + \|v\|_{{\beta}}^2 \| z- w\|_{-\alpha}
   + \|v\|_{{\beta}} \| z^{:2:}-w^{:2:}\|_{-\alpha}
   +\| z^{:3:}- w^{:3:}\|_{-\alpha}.
\end{aligned}
\end{equation}
Then  by  Lemmas \ref{est-P-PN-P}, \ref{Heat-Smooth1}  and the inequality $s^{-a}\lesssim s^{-b}$ with $0\leq a\leq b$ and
 the assumption $\alpha<\beta$, $\frac{2\alpha+\beta}{2}<\gamma'<\frac 1 3$, $\kappa$ small enough,  we  deduce that for $t\in[0,\tau^M\wedge\sigma_N\wedge\nu_N^{M,\epsilon}]$
\begin{equation*}
\begin{aligned}
&\|Y_t-Y_t^N\|_{{\beta}}
 \lesssim_M    \frac{(\log N)^2}{N^\kappa}
 \int_0^t \Big\{(t-s)^{-\frac{\alpha+\beta+\kappa}{2}}
   \|\widetilde{\Psi}(Y_s,\underline {\bar Z}_s) \|_{{-\alpha}}
   +   (t-s)^{-\frac{\kappa}{2}} \|F(Y_s)\|_{{\beta}}
    \Big\} ds \\
 &~~~~+\int_0^t
  \Big\{   (t-s)^{-\frac{\alpha+\beta+\kappa}{2}}
   \|\widetilde{\Psi}(Y_s,\underline {\bar Z}_s)-
          \widetilde{\Psi}(Y_s^N,\underline{\bar Z}_s^N) \|_{{-\alpha}}
   +   (t-s)^{-\frac{\kappa}{2}} \|F(Y_s)-F(Y_s^N)\|_{{\beta}}
   \Big\} ds \\
&\lesssim_M
  \frac{(\log N)^2}{N^\kappa}
 \int_0^t  \big\{   (t-s)^{-\frac{\alpha+\beta+\kappa}{2}}s^{-2\gamma'}+(t-s)^{-\frac{\kappa}{2}} s^{-3\gamma'} \big\} ds\\
  &~~~~  + \|\bar Z-\bar Z^N\|_{\bar {\mathcal L}}
    \int_0^t(t-s)^{-\frac{\alpha+\beta+\kappa}{2}}   s^{-2\gamma'-\frac \kappa 2}  ds\\
  &~~~~ +
     \int_0^t\big\{ (t-s)^{-\frac{\alpha+\beta+\kappa}{2}}   s^{-\gamma'}
     +(t-s)^{-\frac{\kappa}{2}}  s^{-2\gamma'}\big\} \|Y_s-Y_s^N\|_{{\beta}}ds,
 \end{aligned}
\end{equation*}
which, together with
\begin{equation}\label{ineq}
  \int_0^t (t-s)^{-a}s^{-b}\leq t^{1-a-b},
\end{equation}
for $a,b>0$ satisfying $a+b<1$, implies that
\begin{equation*}\label{20-Y-YN}
\begin{aligned}
\|Y_t-Y_t^N\|_{{\beta}}
&\lesssim
  \frac{(\log N)^2}{N^\kappa} t^{1-\frac{\kappa}{2}-3\gamma'}
    + \|\bar Z-\bar Z^N\|_{\bar {\mathcal L}}
     t^{1-\frac{\alpha+\beta+2\kappa}{2}-2\gamma'}\\
&~~~~   +
    \int_0^t\big\{ (t-s)^{-\frac{\alpha+\beta+\kappa}{2}}   s^{-\gamma'}
     +(t-s)^{-\frac{\kappa}{2}}  s^{-2\gamma'}\big\} \|Y_s-Y_s^N\|_{{\beta}}ds.
 \end{aligned}
\end{equation*}
Multiplying by $t^{\gamma'}$  and using the Gronwall's inequality we have  for $t\in[0,\tau^M\wedge\sigma_N\wedge\nu_N^{M,\epsilon}]$
\begin{equation*}
t^{\gamma'} \|Y_t-Y_t^N\|_{{\beta}}
\lesssim_M
  \frac{(\log N)^2}{N^\kappa} t^{1-\frac{\kappa}{2}-2\gamma'}
+  \|\bar Z-\bar Z^N\|_{\bar {\mathcal L}}   t^{1-\frac{\alpha+\beta+2\kappa}{2}-\gamma'},
\end{equation*}
where the constants we omit are independent of $N$. Then we find $C=C(M,T)>0$, which is independent of $N$, such that
\begin{equation}\label{Y-ES}
 \|Y-Y^N\|_{\mathcal{M}_ {\sigma_N\wedge\tau^M\wedge\nu_N^{M,\epsilon}}}
\leq
C\big(
  \frac{(\log N)^2}{N^\kappa}
+  \|\bar Z-\bar Z^N\|_{\bar {\mathcal L}} \big).
\end{equation}
Let $\epsilon$ and $N_0=N_0(M)$ be such that $\|\bar Z-\bar Z^N\|_{\bar {\mathcal L}}<\epsilon$ and
$\epsilon+ \frac{(\log N)^2}{N^\kappa} <\frac 1{ C}$ hold
for any $N>N_0$. Such $\epsilon$ and $N_0$ can be ensured by \eqref{ZTN-ZN-n-ini}. Then
$\|Y-Y^N\|_{ {\mathcal M}_{\sigma_N\wedge\tau^M\wedge\nu_N^{M,\epsilon}}}<1$  for every $N>N_0$,
and, recalling the definition of $\sigma_N^M$, we have $\sigma_N\wedge\tau^M\wedge\nu_N^{M,\epsilon}= \tau^M\wedge\nu_N^{M,\epsilon}$. As a consequence, for every $N>N_0$ we have
\begin{equation}\label{t-est1}
   \sup_{0\leq t\leq \tau^M\wedge\nu_N^{M,\epsilon}}t^{\gamma'} \|Y_t-Y_t^N\|_{{\beta}}<1.
\end{equation}
Besides, by the definition of $\tau^M,\nu_N^{M,\epsilon}$,  we also have for every $N\geq N_0$
\begin{equation}\label{set0}
\begin{aligned}
    &  \sup_{0\leq t\leq \tau^M\wedge\nu_N^{M,\epsilon}} \sup_{n=1,2,3}
    \Big\{t^{\gamma'}\|Y_t\|_{{\beta}},t^{\gamma'}\|Y_t^N\|_{{\beta}},
    t^{\frac {(\alpha+\beta)(n-1)}{2}}\|\bar Z_t^{:n:}\|_{-\alpha}, \\
     &~~~~~~~~~~~~~~~ ~t^{\frac {(\alpha+\beta)(n-1)+n \kappa}{2}}\|(\bar Z_t^N)^{:n:}\|_{-\alpha}\Big\}\leq M+1.
\end{aligned}
\end{equation}
By Theorems \ref{Z-INI} and \ref{est-Y-HOLD-Th} we can deduce that for the above fixed $M,\epsilon$
 \begin{equation} \label{t=T,1}
 \lim_{N\rightarrow\infty}P(\nu_N^{M,\epsilon}=T)=1,
 \end{equation}
 and
 \begin{equation} \label{t=T,2}
 \lim_{M\rightarrow\infty}P(\tau^M=T)=1.
\end{equation}
\bigskip

In the following we consider  pathwise error estimates for  $Y^N,Z^N,Y$ and $Z$ before the stopping times $\tau^M\wedge\nu_N^{M,\epsilon}$ with $\epsilon,M>0$ fixed as above for $N>N_0$. All the  constants may depend on $M$.

\begin{theorem}\label{MAIN-THM1}
Assume the setting in Section \ref{SEC4-1}. Then for $N\geq N_0=N_0(M)$
\begin{equation} \label{PY-Y-LP-LEM}
  \E \Big(\sup_{0\leq t\leq \tau\wedge\nu_N^{\epsilon}}t^{\gamma' p}  \|P_NX_t-X_t^N\|_{{\beta}}^p\Big)
  \lesssim_M (1+N^2)^{-\frac{p \alpha^- }{2}},
\end{equation}
where  $\alpha^-$  denotes $\alpha-\delta$ for every $\delta>0$.
\end{theorem}
\begin{proof}  Since
\[
P_NX_t-X_t^N=P_NY_t-Y_t^N,
\]
  and according to \eqref{mild-2} and \eqref{mild-2-N} with the decomposition in \eqref{Nota-F-De}, we have
\begin{equation*}\begin{aligned}
  & P_NY_t-Y_t^N=\int_0^t P_Ne^{(t-s)A}\Big\{F(Y_s)-F(Y_s^N)+\widetilde{\Psi}(Y_s,\underline{\bar Z}_s)-
   \widetilde{\Psi}(Y_s^N,\underline{\bar Z}_s^N)
           \Big\}ds .
\end{aligned}\end{equation*}
By Lemmas \ref{est-P-PN-P}, \ref{Heat-Smooth1} with $\beta>\alpha$ and the inequality $s^{-a}\lesssim s^{-b}$ with $0\leq a\leq b$ we deduce for  $t\leq \tau^M\wedge\nu_N^{M,\epsilon}$,
\begin{equation}\label{PN-Y715}
\begin{aligned}
&\|P_NY_t-Y_t^N\|_{{\beta}}\\
 \lesssim& \int_0^t\Big\{
    (t-s)^{-\frac{\alpha+\beta+\kappa}{2}}
   \|\widetilde{\Psi}(Y_s,\underline {\bar Z}_s)-\widetilde{\Psi}(Y_s^N,\underline{\bar Z}_s^N) \|_{{-\alpha}}
   +   (t-s)^{-\frac{\kappa}{2}} \|F(Y_s)-F(Y_s^N)\|_{{\beta}}
   \Big\} ds \\
\lesssim &\int_0^t \Big\{\big((t-s)^{-\frac{\alpha+\beta+\kappa}{2}}s^{-\gamma'}+
     (t-s)^{-\frac{\kappa}{2}} s^{-2\gamma'}\big) \|Y_s-Y_s^N\|_{{\beta}}
     + (t-s)^{-\frac{\alpha+\beta+\kappa}{2}} \\
     &~~~~\cdot\big( s^{-2\gamma'}\|\bar Z_s -\bar Z^N_s\|_{{-\alpha}}
     +s^{-\gamma'}\|\bar Z_s^{:2:} -(\bar Z^N_s)^{:2:}\|_{{-\alpha}}
     +\|\bar Z_s^{:3:} -(\bar Z^N_s)^{:3:}\|_{{-\alpha}}\big)\Big\}
   ds,
 \end{aligned}
\end{equation}
where the second inequality follows by the estimates in \eqref{F-PSI2}. We note that in the second inequality of \eqref{PN-Y715}, by \eqref{est-P-PN} we have
\[
\|Y_s-Y_s^N\|_{{\beta}}\lesssim \frac{(\log N)^{2}}{N^{\alpha }}\|Y_s\|_{{\beta+\alpha}}+\|P_NY_s-Y_s^N\|_{{\beta}}.
\]
Following \eqref{PN-Y715}, we further have
\begin{equation*}\label{PN-Y716}
\begin{aligned}
&\|P_NY_t-Y_t^N\|_{{\beta}}
\lesssim \int_0^t \big\{(t-s)^{-\frac{\alpha+\beta+\kappa}{2}}s^{-\gamma'}+
     (t-s)^{-\frac{\kappa}{2}} s^{-2\gamma'}\big\} \|P_NY_s-Y_s^N\|_{{\beta}}ds\\
 &~ +    \int_0^t \Big\{\frac{(\log N)^{2}}{N^{\alpha }}\big((t-s)^{-\frac{\alpha+\beta+\kappa}{2}}s^{-\gamma'}+
     (t-s)^{-\frac{\kappa}{2}} s^{-2\gamma'}\big) \|Y_s\|_{{\beta+\alpha}}
     +
     (t-s)^{-\frac{\alpha+\beta+\kappa}{2}} \\
     &~~~~~~~\cdot\big( s^{-2\gamma'}\|\bar Z_s -\bar Z^N_s\|_{{-\alpha}}
     +s^{-\gamma'}\|\bar Z_s^{:2:} -(\bar Z^N_s)^{:2:}\|_{{-\alpha}}
     +\|\bar Z_s^{:3:} -(\bar Z^N_s)^{:3:}\|_{{-\alpha}}\big)\Big\}
   ds.
 \end{aligned}
\end{equation*}
Similar to  Theorem \ref{est-Y-HOLD-Th}, we deduce that $\E \sup_{0\leq t\leq T}t^{\gamma' p}\|Y_{t}\|^p_{\beta+\alpha}<\infty$ since by the setting in Section \ref{SEC4-1} that $ \frac{2\alpha+\beta}{2}<\gamma'<\frac 1 3$, $\bar\beta:=\beta+\alpha$ satisfies the conditions on $\beta$ in Theorem \ref{est-Y-HOLD-Th}.
Moreover, for $\alpha,\beta,\gamma'$ satisfying $ \frac{2\alpha+\beta}{2}<\gamma'<\frac 1 3$, it also holds that
$1-\frac{n(2\alpha+\beta)}{2}-(3-n)\gamma'>0$, $n=1,2,3$.
Then using Gronwall's inequality and  \eqref{ineq}, we have
\begin{equation*}
\begin{aligned}
  \E \sup_{0\leq t\leq \tau^M\wedge\nu_N^{M,\epsilon}} t^{\gamma' p} &\|P_NY_t-Y_t^N\|_{{\beta}}^p
   \lesssim_M    T^ {(  {1-\frac{\kappa }{2}-2\gamma'} ){p}}  \frac{(\log N)^{2p}}{N^{\alpha p}}
      \E \sup_{0\leq t\leq T}t^{\gamma' p}\|Y_{t}\|^p_{\beta+\alpha}\\
& +   T^{(1-\frac{n(2\alpha+\beta+\kappa)}{2}-(2-n)\gamma')p}
 \E\sup_{0\leq t\leq T}t^{\frac{(\alpha+\beta){(n-1)}+n\alpha}{2} p} \|\bar Z^{:n:}_{t}-(\bar Z^N_{t})^{:n:}\|_{ {-\alpha}}^p
\end{aligned}
\end{equation*}
with $\kappa>0$ sufficiently small. By \eqref{ZTN-ZN-n-ini}  we also have
\[
\E \sup_{0\leq t\leq T}  t^{\frac{(\alpha+\beta){(n-1)}+n\alpha}{2} p}
    \|\bar Z^{:n:}_{t}-(\bar Z^N_{t})^{:n:}\|_{ {-\alpha}}^p \lesssim \frac{(\log N)^{2p}}{N^{\alpha p}}+(1+N^2)^{-\frac{ p\alpha^-}{2}}.
    \]
Hence \eqref{PY-Y-LP-LEM} follows.

\end{proof}

\subsection {Rates of convergence for stochastic Allen-Cahn equations}

Now we present rates of convergence for stochastic Allen-Cahn equations, which is the main result throughout our paper.
\begin{theorem}\label{main-rate}
Assume the setting in  Section \ref{SEC4-1} and let $X_0\in\C^{-\alpha}$ with $\alpha\in(0,2/9)$. Let $X$, $X^N$ denote the solutions to equations \eqref{initial-Eq} and \eqref{app-Eq} on $[0,T]$ with initial values $X_0$ and $P_NX_0$ respectively.  Then for any $\delta>0$ and $\gamma'>\frac{3\alpha}{2} $
\begin{equation} \label{Px-x-LP-T}
 \lim_{N\rightarrow\infty}\mathbb P \Big(\sup_{t\in[0,T]} t^{\gamma'} \|P_NX_t-X_t^N\|_{-\alpha}\gtrsim N^{\delta-{\alpha}}\Big)=0,
\end{equation}
\begin{equation} \label{Px-x-LP-T2}
 \lim_{N\rightarrow\infty}\mathbb P \Big(\sup_{t\in[0,T]}t^{\gamma'} \|X_t-X_t^N\|_{{-\alpha}}\gtrsim N^{\delta-{\alpha}}\Big)=0.
\end{equation}
\end{theorem}

\begin{proof} By similar discussion as in Remark \ref{remgam}, it is enough to prove \eqref{Px-x-LP-T} and \eqref{Px-x-LP-T2} for every $\gamma'\in(3\alpha/2,1/3)$.  Taking the stopping times in Section \ref{SEC4-1} with $\epsilon>0$  fixed, we see that
\begin{equation}\label{1+2pro}
\begin{aligned}
\mathbb P (\sup_{t\in[0,T]} t^{\gamma'} &\|P_NX_t-X_t^N\|_{-\alpha}\gtrsim N^{\delta-{\alpha}})
\leq  \mathbb P (\sup_{t\in[0,\tau^M\wedge\nu_N^{M,\epsilon}]}t^{\gamma'} \|P_NX_t-X_t^N\|_{-\alpha}\gtrsim N^{\delta-{\alpha}})\\
&+\mathbb P (\sup_{t\in(\tau^M\wedge\nu_N^{M,\epsilon},T]} t^{\gamma'} \|P_NX_t-X_t^N\|_{-\alpha}\gtrsim N^{\delta-{\alpha}}).
\end{aligned}
\end{equation}
 On the one hand, by \eqref{t=T,2} for any $\varepsilon>0$ we have $M>0$  sufficiently large such that $P(\tau^M <T)<\varepsilon$, and  for such $M$ fixed, by \eqref{t=T,1} there exists $N_1=N(M)$ ($N_1>N_0$ with $N_0$ given in \eqref{PY-Y-LP-LEM}) such that $P(\nu_N^{M,\epsilon}<T)<\varepsilon$ for all $N\geq N_1$. Then for all $N\geq N_1$,
\[
\mathbb P (\sup_{t\in(\tau^M\wedge\nu_N^{M,\epsilon},T]} t^{\gamma'} \|P_NX_t-X_t^N\|_{-\alpha}\gtrsim N^{\delta-{\alpha}})
\leq \mathbb P(\tau^M\wedge\nu_N^{M,\epsilon}<T)<\varepsilon.
\]
On the other hand,  by \eqref{PY-Y-LP-LEM} and letting $\beta$ close to  $\alpha$ ($\beta>\alpha$) and by using the embedding $\|\cdot\|_{-\alpha}\lesssim\|\cdot\|_{\beta}$ and  Markov's inequality, for the above $M$, the first term on the right hand of \eqref{1+2pro}  tends to zero when $N$ tends to infinity. Then   \eqref{Px-x-LP-T} follows.

 Moreover, we have
  \begin{equation*}\label{f}
     \|X_t-X_t^N\|_{-\alpha}\lesssim \|X_t-P_NX_t\|_{-\alpha}+\|P_NX_t-X^N_t\|_{-\alpha},
 \end{equation*}
 which by \eqref{est-P-PN} for any $\lambda<\alpha$ we get
 \begin{equation*}
   \|X_t-P_NX_t\|_{-\alpha}
  \lesssim  \frac{(\log N)^{2p}}{N^{\lambda p}} \|X_t\|_{-\alpha+\lambda},
\end{equation*}
and similar to \eqref{X=Y+Z-EST}  we have $\E\sup_{t\in[0,T]}t^{\gamma'p}\|X_t\|_{-\alpha+\lambda}^p<\infty$. \eqref{Px-x-LP-T2} holds by choosing $\lambda$ close to $\alpha$.

\end{proof}
\appendix
\section{Appendix} \label{NOISE}
\begin{definition}
\label{WHITE-NOISE}
Let $\{\xi(\phi)\}_{\phi\in L^2(\mathbb R\times \mathbb T^d)}$ be a family of centered Gaussian random variables on a probability space $(\Omega,\mathcal F,\mathbb P)$ such that
\[
\mathbb E(\xi(\phi)\xi(\psi))=\langle\phi,\psi\rangle_{L^2(\mathbb R\times \mathbb T^d)},
\]
for all $\psi,\phi\in L^2(\mathbb R\times \mathbb T^d)$. Then $\xi$ is called a space-time white noise on $\mathbb R\times \mathbb T^d$. We interpret $\xi(\phi)$ as a stochastic integral and write
\[
\int_{\mathbb R\times \mathbb T^d}\psi(t,x)\xi(dt,dx):=\xi(\psi),~~\psi\in L^2(\mathbb R\times \mathbb T^d).
\]
\end{definition}
For any $n\in\mathbb N$, the multiple stochastic integrals (see \cite[Chapter 1]{Nu2006}) on $\mathbb R\times \mathbb T^d$ are defined for all symmetric functions $f$ in $L^2(\mathbb R\times \mathbb T^d)$, i.e. functions such that
\[
 f(z_1,\ldots,z_n)=f(z_{\sigma(1)},\ldots,z_{\sigma(n)}),~~z_i\in \mathbb R\times \mathbb T^d, j=1,2,\ldots,n,
\]
for any permutation $(\sigma(1),\ldots,\sigma(n))$ of $(1,\ldots,n)$. For such a symmetric function $f$ we denote its $n$-th interated stochastic integral by
\[
I_n(f):=\int_{(\mathbb R\times \mathbb T^d)^n}f(z_1,\ldots,z_n)\xi(\otimes_{i=1}^nds_i,\otimes_{i=1}^ndx_i),~z_i=(t_i,x_i)\in\mathbb R\times \mathbb T^d.
\]
\begin{theorem}{\cite[Theorem 1.1.2, Section 1.4]{Nu2006}}
\label{decom-wick}
Let $f$ be any symmetric function in $L^2((\mathbb R\times \mathbb T^d)^n)$. Then
\begin{equation}\label{iTO-EQ}
\mathbb E(I_n(f))^2=n!\|f\|_{L^2((\mathbb R\times \mathbb T^d)^n)}^2
\end{equation}
and
\begin{equation}\label{iTO-iEQ}
\mathbb E|I_n(f)|^p\leq (p-1)^{\frac{np}{2}}(\mathbb E|I_n(f)|^2)^{\frac{p}{2}}
\end{equation}
for every $p\geq 2$.
\end{theorem}
\bigskip
For  symmetric kernels $K_1,K_2:\mathbb Z^2\rightarrow(0,\infty)$, we denote its convolution
\[
  K_1\star K_2(m):=\sum_{l\in\mathbb Z^2} K_1(m-l)K_2(l)
\]
and for $N\in\mathbb N$ we set
\[
  K_1\star_{\leq N} K_2(m):=\sum_{|l|\leq N} K_1(m-l)K_2(l),~~
  K_1\star_{> N} K_2(m):=K_1\star K_2-K_1\star{_{\leq N}} K_2.
\]
For convolutions of the same kernel, we introduce
\begin{equation}\label{n}
   K\star^{1} K:=K,~ K\star^{n} K:= K\star({K\star^{n-1} K}),
\end{equation}
\begin{equation}\label{<Nn}
 K{\star^1}_{\leq N} K:= K, ~K{\star^n}_{\leq N} K:=K\star_{\leq N}({K{\star^{n-1}}_{\leq N} K}),
\end{equation}
where by simple calculation we actually obtain
 \begin{align*}
  K{\star^n}_{\leq N} K(m) =&\sum_{|l_i|\leq N,i=1,\ldots,n-1}  K(m-l_{n-1}) \prod^{n-1}_{i=1} K(l_i-l_{i-1}),
\end{align*}
with  the convention that $l_0 = 0$. Similarly for every $n\geq 2$, we denote
\begin{equation}\label{>Nn}
  K{\star^n}_{> N} K(m):=K\star^n K-K{\star^n}_{\leq N} K.
\end{equation}
Following the technique of \cite[Lemma 10.14]{Hai14}, we have the following estimates.
\begin{lemma}\label{KENEL-ES}
Let $K^\gamma:\mathbb Z^2\rightarrow(0,\infty)$ be a symmetric kernel such that $K^\gamma(m)\lesssim \frac{1}{(1 + |m|^2)^{1-\gamma}}$, $\gamma\in[0,\frac{1}{n})$ for $n\in \mathbb N$.
 {\rm(i)} If $\gamma>0$ then
\begin{equation*}\label{Kn}
\max\big\{K^\gamma \star^{n} K^\gamma (m),~~\sup_{N\geq 1}K^\gamma{\star{^{n}}}_{\leq N} K^\gamma (m)\big\}\lesssim \frac{1}{(1 + |m|^2)^{1-n\gamma}},
\end{equation*}
\begin{equation*}\label{K>Nn}
K^\gamma{\star{^{n}}}_{> N} K^\gamma (m) \lesssim  \left\{
\begin{aligned}
  &\frac{1}{(1 + |m|^2)^{1-n\gamma}},~~\text{if}~|m|\geq N,
\\&\frac{1}{(1 + |N|^2)^{1-n\gamma}},~~\text{if}~|m|< N.
\end{aligned}
\right.\end{equation*}
 {\rm(ii)} If  $\gamma=0$ then
\begin{equation*}\label{Kn}
\max\big\{ K^0\star^{n} K^0 (m),~~\sup_{N\geq 1}K^0{\star{^{n}}}_{\leq N} K^0 (m) \big\}
\lesssim   \frac{1}{(1 + |m|^2)^{1-\epsilon}},
\end{equation*}
\begin{equation*}\label{K>Nn}
K^0{\star{^{n}}}_{> N} K^0 (m) \lesssim \left\{
\begin{aligned}
  &\frac{1}{(1 + |m|^2)^{1-\epsilon}},~~\text{if}~|m|\geq N
\\&\frac{1}{(1 + |N|^2)^{1-\epsilon}},~~\text{if}~|m|< N,
\end{aligned}
\right.\end{equation*}
for every $\epsilon\in(0,1)$.
\end{lemma}
\begin{proof}
The estimates for $K^\gamma\star^{n} K^\gamma$ and $K^\gamma{\star{^{n}}}_{> N} K^\gamma$ with $\gamma\in[0,\frac 1 n)$ were obtained in \cite[Corollary C.3]{TW20182}. The term  $K^\gamma{\star{^{n}}}_{\leq N} K^\gamma$ can be similarly considered following the same procedure. Now let two  symmetric kernels $K_1,K_2:\mathbb Z^2\rightarrow(0,\infty)$ be such that $K_1(m)\lesssim\frac{1}{(1 + |m|^2)^\alpha}$ and $K_2(m)\lesssim\frac{1}{(1 + |m|^2)^\beta}$ with any $\alpha,\beta\in(0,1]$ and $\alpha+\beta-1>0$. We consider the following regions of $\mathbb Z^2$,
\begin{equation*}
\begin{aligned}
 &A_1=\{l:~|l|\leq \frac{|m|}{2}\},~~A_2=\{l:~|l-m|\leq \frac{|m|}{2}\},\\
 &A_3=\{l:~\frac{|m|}{2}\leq |l|\leq 2|m|,|l-m|\geq\frac{|m|}{2}\},~~A_4=\{l:~|l|\geq 2|m|\}.
\end{aligned}
\end{equation*}
Since for every $l\in A_1$ we have $|m-l|\geq |m|-|l|\geq \frac{|m|}{2}$, then for uniform $N\geq 1$
\begin{equation*}
 \begin{aligned}
 \sum_{l\in A_1,|l|\leq N} K_1(m-l)K_2(l)&\lesssim \frac 1{(1+|m|^2)^\alpha}
       \sum_{l\in A_1} \frac 1{(1+|l|^2)^\beta}\\
       &\lesssim\left\{
       \begin{aligned}
       & \frac{(1+|m|^2)^{1-\beta}}{(1+|m|^2)^\alpha}, &\text{if $\beta<1$,}\\
       &\frac{\log |m|\vee 1}{(1+|m|^2)^\alpha}, &\text{if $\beta=1$.}
\end{aligned}
\right.
\end{aligned}
\end{equation*}
For $l\in A_2$, by symmetry we get that for uniform $N\geq 1$
\begin{equation*}
 \sum_{l\in A_2,|l|\leq N} K_1(m-l)K_2(l)
       \lesssim\left\{
       \begin{aligned}
       & \frac{(1+|m|^2)^{1-\alpha}}{(1+|m|^2)^\beta}, &\text{if $\alpha<1$,}\\
       &\frac{\log |m|\vee 1}{(1+|m|^2)^\alpha}, &\text{if $\beta=1$.}
\end{aligned}
\right.
\end{equation*}
For $l\in A_3$ we notice that by the definition of $A_3$, for uniform $N\geq 1$ we have
\begin{equation*}
 \sum_{l\in A_3,|l|\leq N} K_1(m-l)K_2(l)
       \lesssim
       \sum_{|l|\leq 2|m|}   \frac 1{(1+|m|^2)^{\alpha+\beta}}
       \lesssim
         \frac 1{(1+|m|^2)^{\alpha+\beta-1}}.
\end{equation*}
For $l\in A_4$, we have $|m|<\frac {|l|} 2$ and then $|m-l|\geq |l|-|m|\geq \frac {|l|}{2}$, which implies that for uniform $N\geq 1$
\begin{equation*}
 \sum_{l\in A_4,|l|\leq N} K_1(m-l)K_2(l)
       \lesssim
       \sum_{|l|> 2|m|}   \frac 1{(1+|l|^2)^{\alpha+\beta}}
       \lesssim
         \frac 1{(1+|m|^2)^{\alpha+\beta-1}},
\end{equation*}
where the second inequality comes from $\alpha+\beta-1>0$.

Combining all the above and considering that $\log|m|\lesssim(1+|m|^2)^\epsilon$ for $\epsilon>0$ arbitrarily small, we thus obtain that for uniform $N\geq 1$
\begin{equation*}
 K_1 \star _{\leq N}K_2(m)
       \lesssim\left\{
       \begin{aligned}
       & \frac{1}{(1+|m|^2)^{1-\epsilon}}, &\text{if $\alpha,\beta=1$,}\\
       &\frac{1}{(1+|m|^2)^{\alpha+\beta-1}}, &\text{if $\beta<1$ or $\alpha<1$,}
\end{aligned}
\right.
\end{equation*}
for $\epsilon>0$ arbitrarily small. We prove Lemma \ref{KENEL-ES} immediately.

\end{proof}

\end{document}